\documentclass[12pt]{amsart}
\usepackage{amscd,amsmath,amsthm,amssymb,verbatim,enumerate,tikz}
\usepackage{color}
\usepackage{blkarray}
\usepackage{caption}

\numberwithin{equation}{section}
\newtheorem{theorem}{Theorem}[section]
\newtheorem{lemma}[theorem]{Lemma}

\newtheorem{corollary}[theorem]{Corollary}

\theoremstyle{definition}
\newtheorem{construction}[theorem]{Construction}
\newtheorem{definition}[theorem]{Definition}

\newtheorem{remark}[theorem]{Remark}
\newtheorem{example}[theorem]{Example}

\begin{document}

\title[Regularity of toric ideals of graphs]
{Bounds on the regularity of toric ideals of graphs}
\thanks{Version: \today}

\author{Jennifer Biermann}
\address{Department of Mathematics and Statistics, \\
Mount Holyoke College\\
South Hadley, MA 01705}
\email{jbierman@mtholyoke.edu}

\author{Augustine O'Keefe}
\address{
Mathematics Department
Connecticut College
Mathematics Department
270 Mohegan Avenue 
New London, CT 06320}
\email{aokeefe@conncoll.edu}

\author{Adam Van Tuyl}
\address{Department of Mathematics \& Statistics\\
McMaster University \\
Hamilton, ON, L8S 4L8, Canada}
\email{vantuyl@math.mcmaster.ca}



\begin{abstract}
Let $G$ be a finite simple graph.  We give a lower bound for the
Castelnuovo-Mumford regularity of the toric
ideal $I_G$ associated to $G$ in terms of the sizes and number
of induced complete bipartite graphs in $G$.  
When $G$ is a chordal bipartite graph,
we find an upper bound for the regularity of $I_G$ in terms of the
size of the bipartition of $G$.  
We also give a new proof
for the graded Betti numbers of the toric ideal associated
to the complete bipartite graph $K_{2,n}$.
\end{abstract}

\maketitle

\tikzstyle{vertex}=[circle, draw, inner sep=0pt, minimum size=6pt]
\newcommand{\vertex}{\node[vertex]}

\section{Introduction}

The last two decades have seen a significant
dictionary developed between the algebraic invariants in the
graded minimal free resolution of 
the edge ideal of a graph $G$
and the graph-theoretic invariants of
$G$ (e.g., see \cite{HHBook,VBook}).  Inspired by this work,
we wish to work towards a similar 
dictionary between finite
graphs and their associated 
toric ideals.

Given a finite simple  graph $G=(V,E)$ with 
vertex set $V=\{x_1,\dots,x_n\}$ and edge set $E=\{e_1,\dots,e_r\}$,
we abuse notation and define the polynomial rings 
$k[V]=k[x_1,\dots,x_n]$ and $k[E]=k[e_1,\dots,e_r]$ where $k$ is any field. 
Define a monomial map $\pi:k[E]\rightarrow k[V]$ by $e_i\mapsto x_{i_1}x_{i_2}$ where $e_i=\{x_{i_1},x_{i_2}\}
\in E$.  
The kernel of 
$\pi:k[E]\rightarrow k[V]$, denoted $I_G$, is the 
\textit{toric ideal defined by $G$}.
It is well-known that the generators of $I_G$ correspond to
closed even walks in $G$, and in particular,
$I_G$ is a homogeneous ideal generated by binomials (see \cite[Theorem 8.2.2]{VBook} or \cite{OH1}).  
This construction
is a specific instance of the more general construction of toric ideals;
in the general case, the $e_i$'s are mapped to distinct monomials in 
$K[V]$, and the toric ideal is the kernel of the corresponding map
(see \cite[Chapter IV]{P} for more details).

Because $I_G$ is a homogeneous ideal of $R = k[E]$, 
there is a {\it graded minimal free resolution} associated with $I_G$. 
That is, there exists a long exact sequence of the form 
\[0 \rightarrow \bigoplus_{j \in \mathbb{N}} R(-j)^{\beta_{p,j}(I_G)} 
 \rightarrow \bigoplus_{j \in \mathbb{N}} R(-j)^{\beta_{p-1,j}(I_G)} 
\rightarrow \cdots
 \rightarrow \bigoplus_{j \in \mathbb{N}} R(-j)^{\beta_{0,j}(I_G)} 
\rightarrow I_G \rightarrow 0\]
where $R(-j)$ is the graded $R$-module obtained by shifting the 
degrees of $R$ by $j$ and $p \leq r$.  The numbers
$\beta_{i,j}(I_G)$ are the $(i,j)$-th {\it graded Betti numbers} of $I_G$.

Ideally, one would like to describe the $\beta_{i,j}(I_G)$'s 
in terms of combinatorial data of $G$.   Some work in this
direction has been carried out in \cite{CN}.    In this paper, we
focus on the {\it Castelnuovo-Mumford regularity} (or regularity) of 
$I_G$, that is,
\[{\rm reg}(I_G) = \max\{j-i ~|~ \beta_{i,j}(I_G) \neq 0\}.\]

Our project should be
seen within the context of the much broader problem of understanding the regularity of an arbitrary toric ideal;  e.g, see \cite{BPV} for a method to compute the regularity of a toric ideal, and \cite[Theorem 4.5]{S1} for
an upper bound on the regularity of an arbitrary toric ideal.
Motivation to study the regularity of toric ideals is also partially driven by 
the Eisenbud-Goto conjecture which states that the regularity 
of these ideals should be bounded in terms of the degree and codimension
of the projective variety defined by the toric ideal (see
\cite{EG,S1}).

Our first main result is a lower bound on the regularity
of $I_G$ in terms of the presence of induced
subgraphs that are isomorphic to complete bipartite graphs.  Recall
that the {\it complete bipartite graph} $K_{m,n}$ is the graph
on the vertex set $\{x_1,\ldots,x_m,y_1,\ldots,y_n\}$ and edge
set $E = \{\{x_i,y_j\} ~|~ 1 \leq i \leq m, ~~ 1 \leq j \leq n\}$.  
We  show:

\begin{theorem}[Corollary \ref{finalcor}]\label{maintheorem1}
 Let $G$ be a finite simple graph.  Suppose that
$G$ has an induced subgraph $H$ of the form $H = K_{n_1,n_1} \sqcup
\cdots \sqcup K_{n_t,n_t}$ with each $n_i \geq 2$.  Then
\[{\rm reg}(I_G) \geq n_1+n_2 + \cdots + n_t - (t-1).\]
\end{theorem}

\noindent
It is interesting to note that this
result has a similar flavour to a result about edge ideals
in that the presence of certain induced
subgraphs gives information about the regularity of the ideal (see Remark \ref{edgeideal}).
To prove Theorem \ref{maintheorem1} we use work of 
Aramova and Herzog \cite{AH}
that relates the multigraded Betti number $\beta_{i,\alpha}(I_G)$ to
the $i$-th reduced simplicial homology of a simplicial complex $\Gamma(\alpha)$
associated to a fibre (see next section for complete details).

Our second main result is an upper bound for the regularity of toric ideals
of chordal bipartite graphs, that is, bipartite graphs which have no
induced cycles of length six or more.  Using a result
of Ohsugi and Hibi \cite{OH2} that the toric ideal of $I_G$ for this family
has a quadratic Gr\"obner basis, we are able to associate
with $G$ a new graph $H$ which has the property that $I(H)$, the edge ideal
of $H$, satisfies $\beta_{i,j}(I_G) \leq \beta_{i,j}(I(H))$ for all $i,j \geq 0$.
By applying a result of Woodroofe \cite{W} on the regularity of edge ideals, we derive
an upper bound for the regularity of $I(H)$, and consequently, $I_G$.

\begin{theorem}[Theorem \ref{secondtheorem}]\label{maintheorem2}
Let $G$ be a chordal bipartite graph with bipartition $V = \{x_1,\ldots,
x_n\} \cup \{y_1,\ldots,y_m\}$.  Let 
$r = |\{x_i ~|~ \deg x_i = 1\}|$ and $s = |\{y_j ~|~ \deg y_j = 1\}|$.  
Then
\[{\rm reg}(I_G) \leq \min\{n-r,m-s\}.\]
\end{theorem}

In the case that $G = K_{m,n}$, our upper and lower bounds agree,
thus giving ${\rm reg}(I_{K_{m,n}}) = \min\{m,n\}$ (and recovering a special
case of a result of Corso and Nagel \cite{CN}).

In the last section, we use the techniques of the previous
sections to give a new combinatorial proof for the graded 
Betti numbers of the toric ideal of the complete graph $K_{2,n}$. 
Previous proofs used the Eagon-Northcott resolution, which we now avoid.

As a closing comment, $G$ can also be associated to a binomial
ideal via the {\it binomial edge ideal}, that is,  
the ideal 
generated by all binomials of the form $x_iy_j - x_jy_i$
in the polynomial ring $k[x_1,\ldots,x_n,y_1,\ldots,y_n]$ whenever
$\{i,j\}$ is an edge of $G$. 
This
ideal was independently  
introduced in \cite{Hetal,O}.  For this family of binomial ideals, 
the programme to link the combinatorial data of $G$ with the graded 
minimal free resolution is much further advanced;  e.g., \cite{EZ,MM} study the regularity of these
ideals.  However, the toric ideals of this paper
are rarely binomial edge ideals, so our work complements this research.

\noindent
{\bf Acknowledgements.} 
{\em Macaulay2} \cite{Mt} and {\tt CoCoA} \cite{C} 
were used for computer experiments.  We thank Lakehead University, Mount Holyoke College, and 
McMaster University for their hospitality.
The third author acknowledges the support of an NSERC Discovery Grant.  
We would also like to thank Russ Woodroofe for answering some of our
questions.


\section{Preliminaries}

We review the relevant background needed for the remainder of the paper.

A {\it simplicial complex} $\Delta$ on a set $V = \{x_1,\ldots,x_n\}$
is a set consisting of subsets of $V$ such that $\{x_i\} \in \Delta$
for all $i = 1,\ldots,n$, and if $F \in \Delta$ and $G \subseteq F$,
then $G \in \Delta$.  The {\it facets} of $\Delta$ are the maximal
elements of $\Delta$ with respect to inclusion.   We say that a simplicial complex $\Delta$ is generated by a list of faces $\sigma_1, \dots, \sigma_r$ if every face of $\Delta$ is contained in some $\sigma_i$.  In this case we write $\Delta = \langle \sigma_1, \dots, \sigma _r\rangle$.  If $\Delta_1$ and $\Delta_2$ are simplicial complexes on vertex
sets $V_1$ and $V_2$, respectively, then
the {\it join} of $\Delta_1$ and $\Delta_2$, denoted
$\Delta_1 \star \Delta_2$ is the simplicial complex on $V_1 \cup V_2$
where $\Delta_1 \star \Delta_2 = \{ F_1 \cup F_2 ~|~ F_i \in \Delta_i\}.$

The ring $k[V] = k[x_1,\ldots,x_n]$ has a natural $\mathbb{N}^n$-grading
by setting the degree of $x_i$ to be the $i$-th standard basis vector of $\mathbb{N}^n$.
The monomial map 
$\pi: k[E] \rightarrow k[V]$ defined by $e_i \mapsto x_{i_1}x_{i_2}$ where $e_i = \{x_{i_1},x_{i_2}\}$
then induces an $\mathbb{N}^n$-grading on $k[E]$.  
In particular, for any $\alpha = (a_1,\ldots,a_n) \in \mathbb{N}^n$, 
$k[E]_{\alpha} = \pi^{-1}(k[V]_{\alpha})$.  In other words, the 
degree $\alpha$ component consists of those elements of $k[E]$
that map to $k[V]_{\alpha}$.  Note that $k[E]_{\alpha} = (0)$ if
$a_1+\cdots+a_n$ is odd because each $e_i$ is mapped to a monomial of degree
two.  Going forward, we write $\alpha$ for
both $x^{\alpha} = x_1^{a_1}\cdots x_n^{a_n}$ in $k[x_1,\ldots,x_n]$ and 
its $\mathbb{N}^n$-degree.
The 
{\it support} of a monomial $\alpha$ is the 
set 
\[{\rm supp}(\alpha) = \{x_i ~|~ \mbox{$x_i$ divides $\alpha$}\}.\]

For any monomial $\alpha \in k[x_1,\ldots,x_n]$, the {\it fibre} of $\alpha$
is the set of monomials $C_{\alpha} = \{m \in k[E]~|~ \pi(m) = \alpha\}$.
From each fibre, we can construct a simplicial complex:

\begin{definition}\label{Gamma}  Let $\alpha$ be a monomial in $k[x_1, \dots, x_n]$.  
Define $\Gamma(\alpha)$ to be the simplicial complex on the vertex set 
$\{e_1, \dots, e_r\}$  generated by the following faces:
\[
\Gamma(\alpha) = \langle {\rm supp}(w) = \{e_{i_1},\ldots, e_{i_t}\} ~\mid~ w = e_{i_1}^{b_1}\cdots e_{i_t}^{b_t} \in C_\alpha~ \rangle.
\]
\end{definition}

The ideal $I_G$ is a homogeneous
ideal in $k[E]$ with respect to the $\mathbb{N}^n$-grading described above.
As a consequence, $I_G$ has an $\mathbb{N}^n$-graded minimal free resolution.  
The multigraded Betti numbers of $I_G$ are related to the simplicial
complexes $\Gamma(\alpha)$ via the following theorem.

\begin{theorem}[\cite{AH}]\label{toricbetti}
For any monomial $\alpha \in k[x_1,\ldots,x_n]$, and any $i\geq 0$
\[
\beta_{i, \alpha}(I_G) = \dim_k (\widetilde{H}_i(\Gamma(\alpha);k)) ,
\]
where $\widetilde{H}_i(-;k)$ denotes the reduced simplicial
homology with respect to the field $k$.
\end{theorem}

In the statement below, if $G = (V,E)$ is a finite simple graph and 
$W \subseteq V$, then the {\it induced graph on $W$}, denoted $G_W$, is the graph
with vertex set $W$ and edge  set $\{ e \in E ~|~ e \subseteq W\}$.

\begin{theorem}\label{join}
Let $G$ be a finite simple graph on the vertex set $V = \{x_1,\ldots,x_n\}$.
Let $\alpha$ be a monomial in $k[x_1,\ldots,x_n]$, and suppose  that 
$\alpha = \alpha_1\alpha_2$ with ${\rm gcd}(\alpha_1,\alpha_2) = 1$.  If $G_{{\rm supp}(\alpha)} = G_{{\rm supp}(\alpha_1)} \sqcup
G_{{\rm supp(\alpha_2)}}$, then
\[\Gamma(\alpha) = \Gamma(\alpha_1)\star \Gamma(\alpha_2).\]
\end{theorem}

\begin{proof}  Note that ${\rm gcd}(\alpha_1,\alpha_2) = 1$ 
if and only if ${\rm supp}(\alpha_1) \cap {\rm supp}(\alpha_2) = \emptyset$.
If $F = \{e_{i_1}, \cdots, e_{i_s}\} \in \Gamma(\alpha)$ is a facet, then there exists
$w = e_{i_1}^{c_1}\cdots e_{i_s}^{c_s} \in k[E]$ such that $\pi(w) = \alpha$.
Thus for any variable $e_{i_j}$ that divides $w$, if $e_{i_j} = \{x_k,x_l\}$, 
then $\{x_k,x_l\} \in {\rm supp}(\pi(w)) =
{\rm supp}(\alpha)$.  In other words, $e_{i_j}$ is an edge in the induced
graph $G_{{\rm supp}(\alpha)}$.  Because   $G_{{\rm supp}(\alpha)} = 
G_{{\rm supp}(\alpha_1)} \sqcup G_{{\rm supp(\alpha_2)}}$, we have that 
$e_{i_j}$ is an edge in either $G_{{\rm supp}(\alpha_1)}$ or
$G_{{\rm supp(\alpha_2)}}$, but not both.  So, after a possible relabeling,
we can assume that $e_{i_1},\ldots,e_{i_t}$ are all
edges of $G_{{\rm supp}(\alpha_1)}$ and $e_{i_{t+1}},\ldots,e_{i_s}$ are
all edges of $G_{{\rm supp(\alpha_2)}}$.   We thus have 
\[\alpha = \pi(w) = \pi(e_{i_1}^{c_1}\cdots e_{i_t}^{c_t})\pi(e_{i_{t+1}}^{c_{t+1}}
\cdots e_{i_s}^{c_s}).\]
Furthermore,  
${\rm supp}(\pi(e_{i_1}^{c_1}\cdots e_{i_t}^{c_t})) \subseteq {\rm supp}(\alpha_1)$
and ${\rm supp}(\pi(e_{i_{t+1}}^{c_{t+1}}\cdots e_{i_s}^{c_s})) 
\subseteq {\rm supp}(\alpha_2)$.  Because ${\rm supp}(\alpha_1) \cap {\rm supp}(\alpha_2) 
= \emptyset$ and $\alpha = \alpha_1\alpha_2$, we must have
\[\pi(e_{i_1}^{c_1}\cdots e_{i_t}^{c_t})) = \alpha_1  ~~\mbox{and}~~ 
\pi(e_{i_{t+1}}^{c_{t+1}}\cdots e_{i_s}^{c_s}) = \alpha_2.\]
Hence $e_{i_1}\cdots e_{i_t} \in \Gamma(\alpha_1)$ and $e_{i_{t+1}}\cdots e_{i_s}
\in \Gamma(\alpha_2)$, and thus $F \in \Gamma(\alpha_1) \star
\Gamma(\alpha_2)$.

For the reverse containment, let $F_1 \in \Gamma(\alpha_1)$ and $F_2 \in \Gamma(\alpha_2)$
be facets of $\Gamma(\alpha_1)$ and $\Gamma(\alpha_2)$, respectively.   So,
there exists monomials $w_1, w_2 \in k[E]$ such that $F_1 = {\rm supp}(w_1)$
and $F_2 = {\rm supp}(w_2)$ and $\pi(w_1) = \alpha_1$ and $\pi(w_2) = \alpha_2$.
But then $\alpha = \alpha_1\alpha_2 = \pi(w_1)\pi(w_2) = \pi(w)$ where $w = 
w_1w_2$.  So ${\rm supp}(w) \in \Gamma(\alpha)$.  Note that 
${\rm supp}(w_1) \cap {\rm supp}(w_2) = \emptyset$.  Indeed, if $e_i \in 
{\rm supp}(w_1) \cap {\rm supp}(w_2)$, then if $e_i = \{x_k,x_l\}$ we would have
$\{x_k,x_l\} \in {\rm supp}(\pi(w_1)) \cap {\rm supp}(\pi(w_2)) = 
{\rm supp}(\alpha_1) \cap {\rm supp}(\alpha_2) = \emptyset$.
So ${\rm supp}(w) = {\rm supp}(w_1)\cup{\rm supp}(w_2)$, and thus $F_1 \cup F_2 \in \Gamma(\alpha)$.
\end{proof}

\begin{corollary}\label{kunneth1}
With the hypotheses as in Theorem \ref{join}
\[\widetilde{H}_i(\Gamma(\alpha);k) = 
\bigoplus_{j+\ell = i-1} \widetilde{H}_j(\Gamma(\alpha_1);k) 
\otimes
\widetilde{H}_\ell(\Gamma(\alpha_2);k) ~~\mbox{for all $i \geq 0$}.\]
\end{corollary}

\begin{remark}
The above result follows from the K\"unneth formula.
Our formulation is based upon the one found in the thesis
of E. Emtander (see the bottom of page 8 in \cite{E}).
\end{remark}

As a consequence of the above results, lower bounds on ${\rm reg}(I_G)$
can then be found by bounding the regularity of the toric ideals of
induced subgraphs.

\begin{theorem}\label{inducedregularity}
Let $G$ be a graph which contains an induced subgraph of the form
$H = H_1 \sqcup \cdots \sqcup H_t$. Then 
\[
{\rm reg}(I_G) \geq {\rm reg}(I_{H_1}) + \cdots + {\rm reg}(I_{H_t}) 
-t+1.\]
\end{theorem}

\begin{proof}
We first note that for any graph $K$, 
\[\beta_{i,j}(I_K)  = \sum_{\alpha \in \mathbb{N}^n, ~|\alpha| = 2j} \beta_{i,\alpha}(I_K)\]
where $|\alpha|$ denotes $a_1+\cdots+a_n$ for $\alpha = (\alpha_1, \dots, \alpha_n) \in \mathbb{N}^n$.  
Note that since the homomorphism $\pi:K[E] \rightarrow K[V]$ has degree two, the appropriate inner degree of the Betti number is $j$ not $|\alpha| = 2j$.

Relabel the vertices of  $G$ so that
$\{x_{\ell,1},\ldots,x_{\ell,n_\ell}\}$ are the vertices of $H_\ell$ 
for $\ell=1,\ldots,t$.
If $r_\ell = {\rm reg}(I_{H_\ell})$, then there exists an integer $i_\ell$ and
a monomial $\alpha_\ell$ with ${\rm supp}(\alpha_\ell) \subseteq \{x_{\ell,1},\ldots,x_{\ell,n_\ell}\}$ 
such that $\beta_{{i_\ell},\alpha_\ell}(I_{H_\ell}) \neq 0$ and 
$r_\ell = \frac{|\alpha_\ell|}{2}-i_\ell$.  Thus, by Theorem
\ref{toricbetti}
\[\widetilde{H}_{i_\ell}(\Gamma(\alpha_\ell);k) \neq 0 ~~\mbox{for $\ell=1,\ldots,t$}.\]

Set $\alpha = \alpha_1\cdots \alpha_t$.
It follows by repeated use of Theorem \ref{join} that 
\[\Gamma(\alpha) = \Gamma(\alpha_1)\star \Gamma(\alpha_2) \star \cdots \star \Gamma(\alpha_t).\]
Consequently, by Corollary \ref{kunneth1} we will have 
\begin{equation}\label{kunneth}
\widetilde{H}_i(\Gamma(\alpha);k) = \bigoplus_{j_1+\cdots + j_t = i-t+1}
\widetilde{H}_{j_1}(\Gamma(\alpha_1);k) \otimes \cdots \otimes \widetilde{H}_{j_t}(\Gamma(\alpha_t);k).
\end{equation}
If $i = i_1+\cdots + i_t + (t-1)$, then 
by \eqref{kunneth} we have
$\widetilde{H}_i(\Gamma(\alpha);k) \neq 0$.  So
Theorem \ref{toricbetti}
implies $\beta_{i,\alpha}(I_G) \neq 0$.
Thus $\beta_{i,j}(I_G) \neq 0$ with 
\begin{eqnarray*}
2j &= &|\alpha| = |\alpha_1| + \cdots + |\alpha_t| =  2(r_1  + i_1) + \cdots + 2(r_t + i_t).
\end{eqnarray*}
We thus get the desired inequality:
\begin{eqnarray*}
{\rm reg}(I_G) &=& \max\{j-i ~|~ \beta_{i,j}(I_G) \neq 0\} \\ 
& \geq & (r_1+ i_1) + \cdots + (r_t+i_t) - (i_1+\cdots + i_t + (t-1)) \\
& = & {\rm reg}(I_{H_1}) + \cdots + {\rm reg}(I_{H_t}) - t + 1.
\end{eqnarray*}
\end{proof}

We will also need the following relationship between the graded Betti numbers
of an ideal and those of its initial ideal.
See \cite[Theorem 22.9]{P} and \cite[Corollary 22.13]{P} for a proof.

\begin{theorem}\label{initialidealbound}  Fix a monomial order $<$ on $R = k[x_1\ldots,x_n]$.
Let $I$ be a homogeneous ideal of $R$, and let ${\rm in}_<(I)$ denote
the initial ideal of $I$.  Then for all $i,j \geq 0$
\[\beta_{i,j}(I) \leq \beta_{i,j}({\rm in}_<(I)).\]
Furthermore, if ${\rm in}_<(I)$ has a linear resolution, then we
have an equality of all $i,j \geq 0$.
\end{theorem}


\section{A Lower bound on the regularity of $I_G$}

By Theorem \ref{inducedregularity}, we can find
lower bounds for ${\rm reg}(I_G)$  
if we identify induced subgraphs of $G$ whose regularity is known.
We carry out this program by finding a lower bound on the regularity
of toric ideals associated to the complete bipartite graph $K_{n,n}$.  

For this section, we will find it expedient to write the vertex
set of $K_{n,n}$ as
$V = \{x_1, \dots, x_n, y_1, \dots, y_n\}$ and edge set as
$E = \{e_{i,j} = \{x_i,y_j\} ~|~ 1 \leq i \leq n, ~~ 1 \leq j \leq n\}$.   
With this notation, our goal is to prove the following result:

\begin{theorem}\label{Nonzerobetti1}
Let $n \geq 2$ and $\alpha = (x_1\cdots x_ny_1 \cdots y_n)^{n-1}$.
Then $
\beta_{n^2-2n,\alpha}(I_{K_{n,n}}) \neq 0,$
and consequently,
\[{\rm reg}(I_{K_{n,n}} )\geq  n.\]
\end{theorem}

Theorems \ref{inducedregularity} and \ref{Nonzerobetti1}
then combine to give the following main result:

\begin{corollary} \label{finalcor}
 Let $G$ be a finite simple graph.  Suppose that
$G$ has an induced subgraph $H$ of the form $H = K_{n_1,n_1} \sqcup
\cdots \sqcup K_{n_t,n_t}$ with each $n_i \geq 2$.  Then
\[{\rm reg}(I_G) \geq n_1+n_2 + \cdots + n_t - (t-1).\]
\end{corollary}

\begin{remark}\label{edgeideal}
Corollary \ref{finalcor}
has a similar ``flavour'' to a result about edge ideals.
The {\it edge ideal} of a graph $G = (V,E)$
is the ideal $I(G) = \langle x_ix_j ~|~ \{x_i,x_j\} \in E \rangle$.
Katzman \cite{K2} proved  that the induced
matching number of $G$ plus one is a lower bound for ${\rm reg}(I(G))$.  The
induced matching number is the maximum number of pairwise disjoint edges
such that the induced graph on the vertices of the edges in the matching
is precisely the set of disjoint edges.  Since a disjoint edge 
is a $K_{1,1}$, Katzman's result is equivalent to the statement
that if $G$ is a graph that contains an induced subgraph
of the form $H = K_{1,1} \sqcup K_{1,1} \sqcup \cdots \sqcup K_{1,1}$ ($t$ copies), 
then ${\rm reg}(I(G)) \geq t + 1$.
\end{remark}

We first outline our strategy to prove Theorem \ref{Nonzerobetti1}.
We will focus on the simplicial complex $\Gamma(\alpha)$
when $\alpha = (x_1\dots x_ny_1\dots y_n)^{n-1}$.  We
first show how to construct the generators of the Stanley-Reisner ideal $I(\Gamma(\alpha))$. 
Using the Taylor
resolution, we then show that the $\mathbb{N}^{2n}$-graded Betti number
$\beta_{n^2-n,w}(I(\Gamma(\alpha))) \neq 0$ with $w = e_{1,1}\cdots e_{n,n}$. 
By using Hochster's formula, we are then able to translate
this result into a statement about the non-vanishing 
of $\widetilde{H}_i(\Gamma(\alpha);k)$ when $i = n^2-2n$.  
Then Theorem \ref{toricbetti} implies
$\beta_{n^2-2n,\alpha}(I_{K_{n,n}}) \neq 0$.

We begin by describing the generators of the
Stanley-Reisner ideal of $\Gamma(\alpha)$.

\begin{theorem}\label{srideal}
Let $G = K_{n,n}$ with $n \geq 2$.
If $\alpha = (x_1\cdots x_ny_1 \cdots y_n)^{n-1}$, then 
\begin{eqnarray*}
I(\Gamma(\alpha)) & = &
( e_{i,1}e_{i,2}\cdots e_{i,n}
~~|~~ i=1,\ldots,n )  + 
( e_{1,j}e_{2,j}\cdots e_{n,j}
~~|~~ j=1,\ldots,n )
\subseteq k[E].
\end{eqnarray*}
\end{theorem}

\begin{proof}
For a fixed $i \in \{1,\ldots,n\}$, consider the monomial
$e_{i,1}\cdots e_{i,n} \in k[E]$.  We wish to show
that $e_{i,1}\cdots e_{i,n} \in I(\Gamma(\alpha))$, so
it is enough to show that $\{e_{i,1},\ldots,e_{i,n}\}
\not\in \Gamma(\alpha)$.
Suppose instead that 
$\{e_{i,1},\ldots,e_{i,n}\}
\in \Gamma(\alpha)$.  Then there would exist a monomial
$w \in k[E]$ such that $\pi(w) = \alpha$, and furthermore
$\{e_{i,1},\ldots,e_{i,n}\} \subseteq 
{\rm supp}(w)$.
Because $x_i$ appears in each edge $e_{i,j}$, we have 
$x_i^{n}$ divides $\pi(e_{i,1}\cdots e_{i,n})$, which
in turn divides $\pi(w)$.  But the exponent of $x_i$
in $\pi(w)$ is $n-1$, so we get a contradiction.  
Thus $\{e_{i,1},\ldots,e_{i,n}\}
\not\in \Gamma(\alpha)$.  
This argument also works for all generators of the second ideal
on the right hand side of the statement.  This demonstrates one containment.

For the reverse containment,
let $w = e_{i_1,j_1}\cdots e_{i_t,j_t}$ be any squarefree monomial
in $I(\Gamma(\alpha))$.  We will first show the following claim.

\noindent
{\it Claim.} There is no monomial $w$ in $I(\Gamma(\alpha))$ such that $\pi(w) \mid \alpha$.  

\noindent
{\it Proof of the Claim.}
Suppose that such a monomial does exist and let $w$ be a maximal such monomial with respect to divisibility.  If $\pi(w) = \alpha$, then ${\rm supp}(w) \in \Gamma(\alpha)$ and hence $w \notin I(\Gamma(\alpha))$.  Therefore $\pi(w)$ must strictly divide $\alpha$.  More precisely, there 
exists $a_1,\ldots,a_n,b_1,\ldots,b_n$ such that 
\[\pi(w) = x_1^{a_1}\cdots x_n^{a_n}y_1^{b_1}\cdots y_n^{b_n}
\mid \alpha = (x_1\cdots x_ny_1\cdots y_n)^{n-1}\]
with $a_i,b_j \leq n-1$ for all $i$ and $j$, and furthermore,
at least one $a_i < n-1$ or $b_j < n-1$.
In fact, since each $e_{i,j}$ is a mapped to $x_iy_j$, 
the degrees of the $x_i$'s in $\pi(w)$ will equal the degree
of the $y_i$'s.  In other words,
$a_1 + \cdots + a_n = b_1 + \cdots + b_n$, and consequently,
there must be at least one $a_i < n -1$ and at least one 
$b_j < n -1$.  

Because there is an edge $e_{i,j}$ between
$x_i$ and $y_j$, 
\[\pi(we_{i,j}) = x_1^{a_1}\cdots x_i^{a_i+1}\cdots x_n^{a_n}y_1^{b_1}\cdots y_j^{b_j+1}\cdots y_n^{b_n}.\]
Note that $\pi(we_{i,j})$ still divides $\alpha$.  If $\pi(we_{i,j})
= \alpha$, that would mean that ${\rm supp}(we_{i,j})$ is a facet of
$\Gamma(\alpha)$ and consequently, $\{e_{i_1,j_1},\ldots,e_{i_t,j_t}\} \in \Gamma(\alpha)$, contradicting the
fact that $w \in I(\Gamma(\alpha))$.  So $we_{i,j}$ strictly divides
$\alpha$, but this contradicts our choice of $w$.

By the claim, for any squarefree monomial 
$w  = e_{i_1,j_1}\cdots e_{i_t,j_t} \in I(\Gamma(\alpha))$, we have $\pi(w)\nmid\alpha$.  
In particular, there exists an  $x_i$ (or $y_j$) such that $x_i^{n} | \pi(w)$ (or $y_j^{n} |\pi(w)$).  We assume there is an
$x_i$ since the proof for the case $y_j$ is the same.
So, among
$e_{i_1,j_1},\ldots,e_{i_t,j_t}$ there are at least $n$ distinct edges that are adjacent
to $x_i$.  But there are only $n$ distinct edges adjacent
to $x_i$ in $K_{n,n}$, namely, $e_{i,1},\ldots,e_{i,n}$.
So $w$ is divisible by $e_{i,1}\cdots e_{i,n}$, and so belongs
to the ideal on the right side of the statement.  
This now completes the proof.
\end{proof}

We now show that 
$\beta_{2n-2,w}(I(\Gamma(\alpha))) \neq 0$ with $w = e_{1,1}\cdots e_{n,n}$
by appealing to the theory of cellular resolutions.  
For an introduction to cellular resolutions we refer the reader to 
\cite{MS} or \cite{P}.  In particular we require the following definition and result.

\begin{definition}  Let $\Delta$ be a simplicial complex whose faces are labelled by monomials, and let $w$ be a monomial.  Then
\[
\Delta_{<w} = \{\sigma \in \Delta \mid \mbox{ the label of $\sigma$ strictly divides $w$}\}.
\]
\end{definition}

\begin{theorem}[\cite{MS}] \label{CellularResolution} If $I$ is a monomial ideal in a polynomial ring $S$ and $\Delta$ is a cell complex which supports a resolution of the quotient $S/I$, then
\[
\beta_{i, w}(I) = \dim_k (\widetilde{H}_{i-1}(\Delta_{<w};k)) \, .
\]
\end{theorem}

We will use this theorem in conjunction with Taylor's resolution.

\begin{theorem}[\cite{T}]\label{taylor}
Let $I$ be a monomial ideal with $r$ minimal generators in $S$.  Then the simplex on $r$ vertices supports a free resolution of $S/I$.
\end{theorem}

Now recall that $I(\Gamma(\alpha))$ denotes the Stanley-Reisner ideal of $\Gamma(\alpha)$.  By Theorem \ref{srideal}, 
\[
I(\Gamma(\alpha)) = (m_1, \dots, m_n,p_{n+1},\ldots,p_{n+n})
\]
where $m_i = e_{i,1}e_{i,2}\cdots e_{i,n}$ is the product of the edges incident to $x_i$ for $1\leq i \leq n$  and $p_{n+i} = e_{1,i}e_{2,i}
\cdots e_{n,i}$ for $1\leq i \leq n$ is the product of
edges incident to $y_i$.   
Let $\Delta$ be the simplex on the $2n$ vertices
$\{1,\ldots,2n\}$.  By Theorem \ref{taylor}, $\Delta$ supports a (non-minimal) free resolution of $S/I(\Gamma(\alpha))$ where $S = k[E]$.  
Label the faces of $\Delta$ in the usual way.  That is, the face $\sigma \in \Delta$ with $\sigma = \{ s_1, \dots, s_a,t_1,\ldots,t_b\}$ is labeled by the monomial 
\[
m_\sigma = \mbox{lcm}(m_{s_1}, \dots, m_{s_a},p_{t_1},\ldots,p_{t_b}).  
\]
Note that the faces of $\Delta$ naturally have the form $\sigma = \{s_1, \dots, s_a, t_1, \dots, t_b\}$ where $\{s_1, \dots, s_a\}
\subseteq \{1, \dots, n\}$ and $\{t_1, \dots, t_b\} \subseteq \{n+1, \dots, 2n\}$.  

\begin{lemma} \label{FacesDelta<w} Let $w = e_{1,1}e_{1,2}\cdots e_{n,n-1}e_{n,n}$.  The faces of $\Delta_{< w}$ are exactly those of the form $\sigma = \{s_1, \dots s_a, t_1\dots t_b\}$ with $\{s_1, \dots, s_a\} \subseteq \{1, \dots, n\}$ , $\{t_1, \dots, t_b\} \subseteq \{n+1, \dots, 2n\}$, and $1\leq a,b \leq n-1$.
In particular, the facets of $\Delta_{<w}$ are given by
\[
\sigma_{i,j} = \{1, \dots, \hat{i} \dots, n, n+1, \dots, \widehat{n+j}, \dots, 2n\}
~~\mbox{
for all $1 \leq i,j \leq n$.}\]
\end{lemma}

\begin{proof}
Let $\sigma$ be a face of $\Delta_{<w}$.  As noted above, $\sigma$ has the form $\sigma = \{s_1, \dots, s_a, t_1, \dots, t_b\}$ where $\{s_1, \dots, s_a\} \subseteq \{1, \dots, n\}$ and $\{t_1, \dots, t_b\} \subseteq \{n+1, \dots, 2n\}$.  It is left to prove that $a$ and $b$ are both less than or equal to $n-1$.  

Since $\sigma \in \Delta_{<w}$, the label $m_\sigma$ of the face $\sigma$ must strictly divide $w$.  Since $w$ is the product of all the variables in the ring, this is equivalent to saying that there is some variable $e_{i,j}$ which does not divide $m_\sigma$.  Since the minimal generators of $I(\Gamma(\alpha))$ are the products of the edges incident to each vertex in  $K_{n,n}$, the variable $e_{i,j}$ divides exactly two of these minimal generators, $m_i$ and $p_{n+j}$.  So $i$ and $n+j$ do not belong to $\sigma$, i.e., $a,b \leq n-1$.

Conversely, fix $i$ and $j$ such that $1\leq i,j \leq n$.  Then 
$e_{i,j}\nmid m_{\sigma_{i,j}}$, so $\sigma_{i,j} \in \Delta _{<w}$.  
\end{proof}

We will now show that the complex $\Delta_{<w}$ is shellable which will allow us to compute its homology in terms of the shelling order.

\begin{definition} 
A simplicial complex $\Delta$ is \emph{shellable} if there is an ordering of the facets $F_1, F_2, \dots, F_t$ of $\Delta$  such that for all $1\leq i < \ell \leq t$ there exists some $1\leq j <\ell$ and some $x \in F_\ell$ such that $F_i\cap F_\ell \subseteq F_j\cap F_\ell = F_\ell \setminus \{x\}$.
\end{definition}

\begin{theorem}\label{shelling}
Let $w = e_{1,1}\cdots e_{n,n}$  as in Lemma \ref{FacesDelta<w}.  Then $\Delta_{<w}$ is shellable with shelling order $\sigma_{1, n+1}, \sigma_{2, n+1}, \dots, \sigma_{n, n+1}, \sigma_{1, n+2}, \sigma_{2,n+2} \dots, \sigma_{n, n+2},\dots, \sigma_{1, 2n},\sigma_{2,2n}, \dots, \sigma_{n,2n}$.
\end{theorem}

\begin{proof}
Suppose that $\sigma_{i_1,i_2}$ is earlier than $\sigma_{\ell_1,\ell_2}$ in the shelling order.  Then either $i_2<\ell_2$ or $i_2=\ell_2$ and $i_1<\ell_1$.  In the second case we have 
\[
\sigma_{i_1,i_2}\cap\sigma_{\ell_1,\ell_2} = \{1, \dots \hat{i}_1, \dots, \hat{\ell}_1, \dots, n, n+1, \dots, \widehat{n+\ell_2} \dots, 2n\} = \sigma_{\ell_1,\ell_2} \setminus\{i_1\} \, .
\]
On the other hand if $i_2<\ell_2$, then $\sigma_{\ell_1, i_2}$ comes earlier than $\sigma_{\ell_1, \ell_2}$ in the shelling order and 
\[
\sigma_{i_1,i_2}\cap\sigma_{\ell_1,\ell_2} \subseteq \sigma_{\ell_1, i_2} \cap \sigma_{\ell_1, \ell_2} = \sigma_{\ell_1,\ell_2} \setminus \{n+i_2\} \, .
\]
Therefore $\Delta_{<w}$ is shellable under the given shelling order.  
\end{proof}

\begin{corollary}\label{NonzeroBetti}
Let $G = K_{n,n}$ with $n \geq 2$.
If $\alpha = (x_1\cdots x_ny_1 \cdots y_n)^{n-1}$, then 
\[
\beta_{2n-2, w}(I(\Gamma(\alpha)) \neq 0
~~\mbox{where $w = e_{1,1}\cdots e_{n,n}$.}
\]
\end{corollary}

\begin{proof}
First note that $\Delta_{<w}$ is pure of dimension $2n-3$.  In the shelling order of Theorem \ref{shelling}, the intersection of the last facet in the shelling order with the earlier facets is the entire boundary of the last facet.  This means that the simplicial complex $\Delta_{<w}$ has non-zero homology in dimension $2n-3$ (see Theorems 4.1 and 4.2 of \cite{BW}).  Then by Theorem \ref{CellularResolution}, $\beta_{2n-2, w}(I(\Gamma(\alpha))) \neq 0$. 
\end{proof}

To complete our proof
of Theorem \ref{Nonzerobetti1}, 
we require Hochster's formula which relates
the (multi)-graded Betti numbers of a simplicial complex $\Delta$
to the reduced simplicial homology groups of $\Delta$ and its induced
subcomplexes.

\begin{theorem}[Hochster's formula \cite{H}] \label{simplicialbetti}
 Let $\Delta$ be a simplicial complex on the vertex set $V$ and let $I(\Delta)$ 
be its Stanley-Reisner ideal.  Then 
\[
\beta_{i,j}(I(\Delta)) = 
\sum_{|W| = j, ~~W\subseteq V} \dim_k(\widetilde{H}_{j-i-2}(\Delta_W;k))
\]
where $\Delta_W$ is the restriction of $\Delta$ to the vertex set $W$.
In particular, if $|V| = m$ then
\[
 \beta_{m-i-2,m}(I(\Delta)) = \dim_k (\widetilde{H}_i(\Delta;k)).
\]
\end{theorem}

We are now ready to prove Theorem \ref{Nonzerobetti1} using the above results.

\begin{proof}(of Theorem \ref{Nonzerobetti1})
Let $\alpha = (x_1\cdots x_ny_1 \cdots y_n)^{n-1}$ and let $w =  e_{1,1}\cdots e_{n,n}$.  By Corollary \ref{NonzeroBetti} we know that
the  $\mathbb{N}^{2n}$-graded Betti number $\beta_{2n-2, w}(I(\Gamma(\alpha)) \neq 0$, and therefore, the 
$\mathbb{N}$-graded Betti number $\beta_{2n-2, n^2}(I(\Gamma(\alpha)) \neq 0$.  Hochster's formula (Theorem \ref{simplicialbetti}) then
implies that  $\dim_k (\widetilde{H}_{n^2-2n}(\Gamma(\alpha);k)) \neq 0$.  So by Theorem \ref{toricbetti}, $\beta_{n^2-2n, \alpha}(I_{K_{n,n}}) \neq 0$.  The multidegree $\alpha$ corresponds to degree $n^2-n$ elements in the toric ideal and so $\beta_{n^2 - 2n, n^2-n}(I_{K_{n,n}}) \neq 0$.  Therefore $\mbox{reg}(I_{K_{n,n}}) \geq n^2-n-(n^2-2n) = n$.
\end{proof}

\begin{example}  We illustrate some of the ideas of this section for the 
graph $G = K_{3,3}$.  Let $\alpha = (x_1x_2x_3y_1y_2y_3)^2$.  By Theorem \ref{srideal}, the Stanley-Reisner ideal of $\Gamma(\alpha)$ is 
\[I(\Gamma(\alpha)) = (e_{1,1}e_{1,2}e_{1,3},
e_{2,1}e_{2,2}e_{2,3},e_{3,1}e_{3,2}e_{3,3},
e_{1,1}e_{2,1}e_{3,1},e_{1,2}e_{2,2}e_{3,2},e_{1,3}e_{2,3}e_{3,3}
) \]
in the ring $k[E] =k[e_{i,j}~|~ 1 \leq i,j \leq 3]$.  The Betti table of $I(\Gamma(\alpha))$ is 
\[\begin{tabular}{c|ccccc}
&0&1&2&3&4\\
\hline
0&-&-&-&-&-\\
1&-&-&-&-&-\\
2&-&-&-&-&-\\
3&6&-&-&-&-\\ 
4&-&9&-&-&-\\
5&-&6&18&9&1\\
\end{tabular}.\]
We can see that the Betti number $\beta_{4,9}(I(\Gamma(\alpha))) \neq 0$ which means that $\dim_k(\widetilde H_3(\Gamma(\alpha);k)) \neq 0 $.  
By Theorem \ref{toricbetti}, this means that
$\beta_{3,\alpha}(I_{K_{3,3}}) \neq 0$, which implies that
$\beta_{3,6}(I_{K_{3,3}}) \neq 0$.  Indeed, the Betti
table of $I_{K_{3,3}}$ is
\[\begin{tabular}{c|cccc}
&0&1&2&3\\
\hline
0&-&-&-&-\\
1&-&-&-&-\\
2&9&16&9&-\\
3&-&-&-&1\\
\end{tabular}.\]
In this example, we see that ${\rm reg}(I_{K_{3,3}}) = 3$.  In fact,
as we will show in the next section, the bound ${\rm reg}(I_{K_{n,n}})
\geq n$ is actually an equality.

\end{example}


\section{Upper bounds for chordal bipartite graphs}

In this section we present an upper bound on the regularity of toric ideals associated to a special class of bipartite graphs.

\begin{definition}
A graph is called \emph{chordal bipartite} if $G$ is a bipartite graph and $G$ has no induced subgraphs which are cycles of length six or larger.  
\end{definition}

The motivation behind restricting to this class of graphs is the following theorem by Ohsugi and Hibi.

\begin{theorem}[\cite{OH1,OH2}] \label{chordal bipartite} Let $G$ be a bipartite graph. Then the following are equivalent:
	\begin{enumerate}[(i)]
	\item $G$ is a chordal bipartite graph.
	\item The toric ideal $I_G$ has a Gr\"obner basis consisting of quadratic binomials.
	\end{enumerate}
\end{theorem}

To prove Theorem \ref{chordal bipartite}, 
Ohsugi and Hibi \cite{OH1,OH2} give an explicit description of the monomial ordering which gives rise to a quadratic Gr\"obner basis for $I_G$.  As this ordering will be important in our work, we take the time to define it here.

\begin{construction}[\cite{OH2}]\label{gb} 
For a chordal bipartite graph $G = (V,E)$ with bipartition
$V = V_1 \cup V_2$, define $A_G$ to be the matrix with columns indexed by $V_2= \{y_1, \dots, y_m\}$ and rows indexed by $V_1= \{x_1,\dots,x_n\}$ and the $(i,j)^\text{th}$ entry  given by 
	\[
	(A_G)_{i,j}=\left\{\begin{array}{ll} 1 & \text{if}~\{x_i,y_j\}\in E\\
	0 & \text{otherwise.}\end{array}\right.
	\]
The vertices in $V_1$ and $V_2$ can be relabeled so that the rows and columns of $A_G$ are 
largest to smallest in the reverse lexicographic order from left to right and from top to bottom (for details see \cite{OH2}). 
Using this relabeling,
let $e_{i,j}$ denote the edge $\{x_i, y_j\}$ and order the
variables of $K[E]$ as follows: $e_{1,1}< e_{1,2}< \dots< e_{1,m}< e_{2,1}< \dots< e_{2,m}<\dots< e_{n,m}$.  Using the reverse lexicographic order with the variables in this order gives a quadratic Gr\"obner basis of $I_G$ in $k[E]$.  
\end{construction}

Embedded in Oshugi and Hibi's proof of Theorem \ref{chordal bipartite} is the following fact which 
we record as a lemma.

\begin{lemma}\label{submatrixlemma}
Let $G$ be a chordal 
bipartite graph.  The matrix $A_G$ as ordered in Construction \ref{gb} has
no induced submatrix of form $\begin{bmatrix} 1& 1 \\ 1 & 0 \end{bmatrix}$.
\end{lemma}

Ohsugi and Hibi go on to show that the monomial order of Construction \ref{gb} gives a Gr\"obner basis whose elements are the binomials of the form $e_{a,d}e_{b,c}-e_{a,b}e_{c,d}$ where $a<c$, $b<d$ and the submatrix of $A_G$ given by rows $a$ and $c$ and columns $b$ and $d$ is 
$\begin{bmatrix} 1 & 1 \\ 1 & 1 \end{bmatrix}$.
Furthermore, the initial ideal of $I_G$ under this order is:
\[
{\rm in}_<(I_G) = \left\langle e_{a,d}e_{b,c} \left|
\begin{tabular}{c}
$a<c$, $b<d$, \mbox{ the submatrix of $A_G$ given by rows}\\ 
\mbox{$a$ and $c$ and columns $b$ and $d$ is
$\begin{bmatrix}
1 & 1 \\ 1 & 1 
\end{bmatrix}$}
\end{tabular}
\right\rangle\right..
\]
In other words, ${\rm in}_<(I_G)$ can be determined directly
from  $A_G$ by identifying all submatrices of the  form
 $\begin{bmatrix} 1 & 1 \\ 1 & 1 \end{bmatrix}$,
 and then the indices corresponding to the anti-diagonal of this matrix give us a generator
 of ${\rm in}_<(I_G)$.  Note that ${\rm in}_<(I_G)$ is a quadratic square-free
 monomial ideal, so we can also view it as the edge ideal of some graph.  We formalize
 this idea:
 
 \begin{construction}\label{graphconstruction}
 Let $G = (V_G,E_G)$ be a chordal bipartite graph, and let $A_G$ be constructed
 as in Construction \ref{gb}.  Let $H = (V_H,E_H)$ be the graph with the vertex set
 $V_H = \{e_{i,j} \mid (A_G)_{i,j} = 1 \}$ and edge set 
\[E_H = 
 \left\{\{e_{a,d},e_{b,c}\} ~\left|~  
\begin{array}{c}
a<c, b<d, \mbox{ the submatrix of $A_G$ given by rows} \\
\mbox{$a$ and $c$ and columns $b$ and $d$ is } 
\begin{bmatrix} 1 & 1 \\ 1 & 1 \end{bmatrix} \\
\end{array}
\right\}\right..
\]
Equivalently, let $H$ be the graph with edge ideal $I(H) = {\rm in}_<(I_G)$.
\end{construction}

\begin{example}\label{positionex}
We illustrate the above ideas with the chordal
bipartite graph $G$:
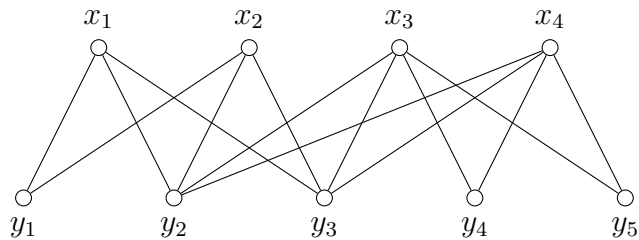
\[\begin{tikzpicture}
          \vertex (x_1) at (1,2)  [label=above:$x_1$] {};
          \vertex (x_2) at (3,2)   [label=above:$x_2$] {};
          \vertex (x_3) at (5,2)  [label=above:$x_3$] {};
          \vertex (x_4) at (7,2)  [label=above:$x_4$] {};
          \vertex (y_1) at (0,0)  [label=below:$y_1$] {};
          \vertex (y_2) at (2,0)  [label=below:$y_2$] {};
          \vertex (y_3) at (4,0)  [label=below:$y_3$] {};
          \vertex (y_4) at (6,0)  [label=below:$y_4$] {};
          \vertex (y_5) at (8,0)  [label=below:$y_5$] {};
	\path
		(x_1) edge (y_1)
                (x_1) edge (y_2)
                (x_1) edge (y_3)
                (x_2) edge (y_1)
                (x_2) edge (y_2)
                (x_2) edge (y_3)
                (x_3) edge (y_2)
                (x_3) edge (y_3)
                (x_3) edge (y_4)
                (x_3) edge (y_5)
	        (x_4) edge (y_2)
                (x_4) edge (y_3)
                (x_4) edge (y_4)
                (x_4) edge (y_5) 
        ;   
\end{tikzpicture}\]
The vertex set is $V = \{x_1,x_2,x_3,x_4,y_1,y_2,y_3,y_4,y_5\}$.  For this graph
$G$, the matrix $A_G$ is 
\[A_ G = \begin{bmatrix}
1 & 1 & 1 & 0 & 0\\
1 & 1 & 1 & 0 & 0 \\
0 & 1 & 1 & 1 & 1 \\
0 & 1 & 1 & 1 & 1 
\end{bmatrix}.\]
Note that under this labeling of the vertices, the rows (respectively
the columns), are ordered from largest to smallest with respect
to the reverse lexicographical order from top to bottom
(respectively from left to right).

We construct the graph $H$ from $A_G$ as follows.  
Replace each 1 by a vertex (keeping the matrix like structure, 
i.e., vertex $e_{i,j}$ is position $(i,j)$) and remove all the zeroes.  
For every $2 \times 2$ submatrix consisting of only ones in $A_G$, we join the two
vertices corresponding to the anti-diagonal.  So, in our example,
our graph $H$ has the form:
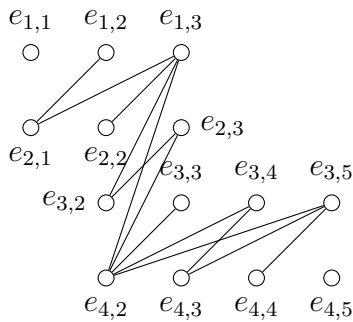
\[\begin{tikzpicture}
          \vertex (e11) at (0,3)  [label=above:$e_{1,1}$] {};
          \vertex (e12) at (1,3)  [label=above:$e_{1,2}$] {};
          \vertex (e13) at (2,3)  [label=above:$e_{1,3}$] {};
          
          \vertex (e21) at (0,2)  [label=below:$e_{2,1}$] {};
          \vertex (e22) at (1,2)  [label=below:$e_{2,2}$] {};
          \vertex (e23) at (2,2)  [label=right:$e_{2,3}$] {};

          \vertex (e32) at (1,1)  [label=left:$e_{3,2}$] {};
          \vertex (e33) at (2,1)  [label=above:$e_{3,3}$] {};
          \vertex (e34) at (3,1)  [label=above:$e_{3,4}$] {};
          \vertex (e35) at (4,1)  [label=above:$e_{3,5}$] {};

          \vertex (e42) at (1,0)  [label=below:$e_{4,2}$] {};
          \vertex (e43) at (2,0)  [label=below:$e_{4,3}$] {};
          \vertex (e44) at (3,0)  [label=below:$e_{4,4}$] {};
          \vertex (e45) at (4,0)  [label=below:$e_{4,5}$] {};
	\path
		(e12) edge (e21)
                (e13) edge (e21)
                (e13) edge (e22)
                (e13) edge (e32)
                (e13) edge (e42)
                (e23) edge (e32)
                (e23) edge (e42)
                (e33) edge (e42)
                (e34) edge (e42)
                (e34) edge (e43)
                (e35) edge (e42)
                (e35) edge (e43)
                (e35) edge (e44)
        ;   
\end{tikzpicture}\]
\end{example}

Note that if we draw the graph $H$ using the matrix $A_G$ as we did in the above example,
we can view every edge in $H$ has 
having an {\it upper-right endpoint} and a 
{\it lower-left endpoint}.
We will use this terminology in the proof below.

Because $I(H) = {\rm in}_<(I_G)$,
if we can bound ${\rm reg}(I(H))$ then a bound on ${\rm reg}(I_G)$ will follow
from  Theorem \ref{initialidealbound}.  To bound ${\rm reg}(I(H))$ we require
a result of Woodroofe bounding the regularity of an edge ideal in terms of co-chordal subgraphs of the graph. 

\begin{definition}
A graph is called \emph{chordal} if it has no induced subgraphs which are cycles of length greater than three.  A graph is called \emph{co-chordal} if its complement is chordal. 
A \emph{co-chordal cover} of a graph $H$ is a set of co-chordal subgraphs $H_1, \dots, H_t$ of $H$ such that $E_H = \bigcup_{i=1}^tE_{H_i}$.  The \emph{co-chordal cover number} of $H$, denoted cochord($H$), is the smallest size of a co-chordal cover of $H$.
\end{definition}

\begin{theorem}[\cite{W}]\label{co-chordal}
Given a graph $H$ with edge ideal $I(H)$, we have
\[
{\rm reg}(I(H)) \leq {\rm cochord}(H) + 1 .
\]
\end{theorem}

We now come to the main result of this section.

\begin{theorem}\label{secondtheorem}
Let $G$ be a chordal bipartite graph with bipartition $V = \{x_1,\ldots,
x_n\} \cup \{y_1,\ldots,y_m\}$.  Let 
$r = |\{x_i ~|~ \deg x_i = 1\}|$ and $s = |\{y_j ~|~ \deg y_j = 1\}|$.  
Then
\[{\rm reg}(I_G) \leq \min\{n-r,m-s\}.\]
\end{theorem}

\begin{proof}
If $G'$ denotes the graph obtained from $G$ by removing all 
the vertices of degree one, then it follows that 
the toric ideals associated to $G'$ and $G$ are the same.
So, without loss of generality, we can
assume that $G$ has no vertices of degree one, and that
$n \leq m$.  

Construct $H$ from $G$ as in Construction
\ref{graphconstruction}.   Because ${\rm in}_<(I_G) = I(H)$, 
to prove the bound ${\rm reg}(I_G) \leq n$, by Theorem 
\ref{initialidealbound} it suffices
to show
that ${\rm reg}(I(H)) \leq n$.
To achieve this goal, we will make use use of 
Theorem \ref{co-chordal}.  Consequently, it suffices
to produce a set $n-1$ co-chordal subgraphs of $H$ which cover $H$.

For each $1\leq i \leq n-1$, define $H_i$ to be the subgraph
of $H$ with vertex set
\[V_{H_i} = \{ e_{a,b} ~|~ e_{a,b} \in V_H ~\mbox{and}~~ a \geq i\}\]
and edge set
\[
E_{H_i} = \{\{e_{i,j},e_{k,\ell}\} \in E_H ~\mid~ 1 \leq j \leq m, ~~ k>i, ~~ \ell<j\}.
\]
That is, $H_i$ consists of all the edges of $H$ whose 
upper-right endpoint is a vertex of the form $e_{i,j}$.   So, if 
the vertices are positioned as in Example \ref{positionex}, i.e., 
vertex $e_{i,j}$ is in position $(i,j)$ where $(A_G)_{i,j} =1$, then
then graph $H_i$ can be visualized as the graph where every
upper-right endpoint is on the $i$-th row.  

Since every edge will have an upper-right endpoint of the form
$e_{a,b}$ with $1 \leq a \leq n-1$, it follows that the $H_i$'s
partition $H$.  To finish the proof, it suffices to show that 
each $H_i$ is a co-chordal graph.  

Fix some $1 \leq i \leq n-1$, and consider $H_i$.  
Every
edge in $H_i$ has its upper-right endpoint among 
$\{e_{i,1},\ldots,e_{i,m}\}$.  (Note that not all of these vertices
may appear as upper-right endpoints.  For example, if $(A_G)_{i,j} = 0$, then the vertex $e_{i,j}$ does not even appear in $H$.)  
Every lower-left endpoint must be among the the
set 
\[\{e_{i+1,1},\ldots,e_{i+1,m},e_{i+2,1},\ldots,e_{i+2,m},
\ldots,e_{n,1},\ldots,e_{n,m}\}.\]

Let $V_1 = V_{H_i} \cap \{e_{i,1},\ldots,e_{i,m}\}$
and $V_2 = V_{H_i} \cap \{e_{i+1,1},\ldots,e_{i+1,m},\ldots,e_{n,1},\ldots,e_{n,m}\}.$
By our construction of $H_i$, $V_1$ and $V_2$ are independent
sets, that is, there are no edges with both endpoints in
both $V_1$, respectively, $V_2$.  
Consequently, in $H_i^c$ (the complement of $H_i$) we have
a clique on  the vertices of $V_1$ and a clique on the vertices of $V_2$.  

We now show that any cycle of length $\geq 4$ in $H_i^c$ must
have a chord.  For any cycle of length $\geq 5$, at least
three of the vertices must be among either $V_1$ or $V_2$.
But since the induced graphs on $V_1$ and $V_2$ in $H_i^c$ are cliques,
these three vertices are all mutually adjacent,
and thus the cycle has a chord.

So, now consider any cycle of length four.  It must have exactly
two vertices in $V_1$ and exactly two vertices in $V_2$.  If not,
it would have at least three vertices in $V_1$ or $V_2$, and as above,
these three vertices would be mutually adjacent.
Let us say that these four vertices are $e_{i,a}, e_{i,b}, e_{j,k},$
and $e_{r,s}$.  Without loss of generality, we can assume that
$a < b$.   

Suppose that we have an induced four cycle on $\{e_{i,a},e_{i,b},e_{j,k},e_{r,s}\}$. Because $e_{i,a}e_{i,b}
\in H_i^c$, the vertex $e_{i,a}$ adjacent to exactly one of 
$e_{j,k}$ and $e_{r,s}$ in $H_i^c$.  Say that $e_{i,a}e_{j,\ell}
\in H_i^c$.  Thus $e_{i,b}$ is adjacent to $e_{r,s}$ in $H_i^c$,
but  $e_{i,a}e_{r,s} \not\in H_i^c$.  But then
$e_{i,a}e_{r,s} \in H_i$, so $s < a$ since $e_{r,s}$ must be a 
lower-left endpoint.  But for $e_{i,a}e_{r,s}$ to be an
edge of $H_i$, and also $H$, the submatrix of $A_G$
given by rows $i$ and $r$ and columns $s$ and $a$ must be
$\begin{bmatrix}1 & 1 \\ 1 & 1 \end{bmatrix}$.  
Thus, in the matrix $A_G$ the submatrix given by rows
$i$ and $r$ and columns $s,a,$ and $b$ has the form
\[
\begin{blockarray}{cccc}
 & s & a & b \\
\begin{block}{c[ccc]}
  i & 1 & 1 & 1 \\
  j & 1 & 1 & \star\\
\end{block}
\end{blockarray}
 \]

By Lemma \ref{submatrixlemma}, the value of $\star$ 
must be $1$.  But this means $e_{i,b}e_{r,s}$ is also 
an edge of $H$, and consequently, the edge 
$e_{i,b}e_{r,s} \not\in H_i^c$.
But this contradicts the fact that $e_{i,b}e_{r,s} \in H_i^c$.
So $H_i^c$ has no induced four cycle.  Consequently, each $H_i$ is a
co-chordal subgraph.
\end{proof}

\begin{example}\label{subgraphex}
We return to our previous example.
Since $n=4$, if we use the notation of the above proof,
the subgraphs $H_1,H_2$, and $H_3$ of $H$ are: 
\vspace{.5cm}

\begin{center}
\begin{minipage}{0.33\textwidth}
\centering
\begin{tikzpicture}
          \vertex (e11) at (0,3)  [label=below:$e_{1,1}$] {};
          \vertex (e12) at (1,3)  [label=above:$e_{1,2}$] {};
          \vertex (e13) at (2,3)  [label=right:$e_{1,3}$] {};
          
          \vertex (e21) at (0,2)  [label=below:$e_{2,1}$] {};
          \vertex (e22) at (1,2)  [label=below:$e_{2,2}$] {};
          \vertex (e23) at (2,2)  [label=right:$e_{2,3}$] {};

          \vertex (e32) at (1,1)  [label=left:$e_{3,2}$] {};
          \vertex (e33) at (2,1)  [label=below:$e_{3,3}$] {};
          \vertex (e34) at (3,1)  [label=below:$e_{3,4}$] {};
          \vertex (e35) at (4,1)  [label=below:$e_{3,5}$] {};

          \vertex (e42) at (1,0)  [label=below:$e_{4,2}$] {};
          \vertex (e43) at (2,0)  [label=below:$e_{4,3}$] {};
          \vertex (e44) at (3,0)  [label=below:$e_{4,4}$] {};
          \vertex (e45) at (4,0)  [label=below:$e_{4,5}$] {};
	\path
		(e12) edge (e21)
                (e13) edge (e21)
                (e13) edge (e22)
                (e13) edge (e32)
                (e13) edge (e42)
        ;   
\end{tikzpicture}
\captionof{figure}{Graph $H_1$}
\end{minipage}\hfill
\begin{minipage}{0.33\textwidth}
\centering

\begin{tikzpicture}
          \vertex (e11) at (0,3)  [label=below:$e_{1,1}$] {};
          \vertex (e12) at (1,3)  [label=below:$e_{1,2}$] {};
          \vertex (e13) at (2,3)  [label=right:$e_{1,3}$] {};
          
          \vertex (e21) at (0,2)  [label=below:$e_{2,1}$] {};
          \vertex (e22) at (1,2)  [label=below:$e_{2,2}$] {};
          \vertex (e23) at (2,2)  [label=right:$e_{2,3}$] {};

          \vertex (e32) at (1,1)  [label=below:$e_{3,2}$] {};
          \vertex (e33) at (2,1)  [label=below:$e_{3,3}$] {};
          \vertex (e34) at (3,1)  [label=below:$e_{3,4}$] {};
          \vertex (e35) at (4,1)  [label=below:$e_{3,5}$] {};

          \vertex (e42) at (1,0)  [label=below:$e_{4,2}$] {};
          \vertex (e43) at (2,0)  [label=below:$e_{4,3}$] {};
          \vertex (e44) at (3,0)  [label=below:$e_{4,4}$] {};
          \vertex (e45) at (4,0)  [label=below:$e_{4,5}$] {};
	\path
                (e23) edge (e32)
                (e23) edge (e42)
        ;   
\end{tikzpicture}
\captionof{figure}{Graph $H_2$}
\end{minipage}
\begin{minipage}{0.33\textwidth}
\centering
\begin{tikzpicture}
          \vertex (e11) at (0,3)  [label=below:$e_{1,1}$] {};
          \vertex (e12) at (1,3)  [label=below:$e_{1,2}$] {};
          \vertex (e13) at (2,3)  [label=right:$e_{1,3}$] {};
          
          \vertex (e21) at (0,2)  [label=below:$e_{2,1}$] {};
          \vertex (e22) at (1,2)  [label=below:$e_{2,2}$] {};
          \vertex (e23) at (2,2)  [label=right:$e_{2,3}$] {};

          \vertex (e32) at (1,1)  [label=below:$e_{3,2}$] {};
          \vertex (e33) at (2,1)  [label=above:$e_{3,3}$] {};
          \vertex (e34) at (3,1)  [label=above:$e_{3,4}$] {};
          \vertex (e35) at (4,1)  [label=above:$e_{3,5}$] {};

          \vertex (e42) at (1,0)  [label=below:$e_{4,2}$] {};
          \vertex (e43) at (2,0)  [label=below:$e_{4,3}$] {};
          \vertex (e44) at (3,0)  [label=below:$e_{4,4}$] {};
          \vertex (e45) at (4,0)  [label=below:$e_{4,5}$] {};
	\path
                (e33) edge (e42)
                (e34) edge (e42)
                (e34) edge (e43)
                (e35) edge (e42)
                (e35) edge (e43)
                (e35) edge (e44)
        ;   
\end{tikzpicture}

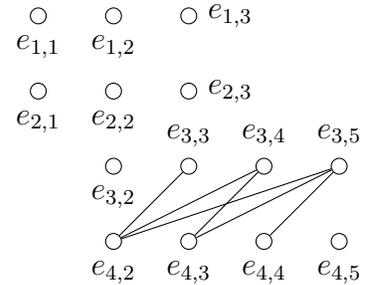
\captionof{figure}{Graph $H_3$}
\end{minipage}
\vspace{.25cm}
\end{center}

\noindent
Each of the graphs $H_1, H_2$, and $H_3$ are co-chordal,
that is, their complement is a chordal graph.  So
${\rm reg}(I_G) \leq 3+1 = 4$.  By using {\it Macaulay2}, we actually 
have ${\rm reg}(I_G) = 4$.
\end{example}

In some cases, the upper bound of this section agrees with the lower bound
of Corollary \ref{finalcor}, as in the case when $G = K_{n,m}$:
\begin{corollary}
Let $G = K_{n,m}$ be a complete bipartite graph with $n\leq m$.  Then
\[{\rm reg}(I_G) = n.\]
\end{corollary}

\begin{remark}
The above corollary is also a special case of result of 
\cite[Proposition 5.7]{CN} since 
the graph $K_{n,m}$ is an example of a Ferrers graph.
However, chordal bipartite graphs are 
not necessarily Ferrers graphs;  the example of Example \ref{positionex} is not
a Ferrers graph.
\end{remark}


\section{Betti numbers of the toric ideals $I_{K_{2,d}}$}

We give a new proof for the graded Betti numbers of the toric
ideals $I_{K_{2,d}}$.  The matrix $A_G$
used to construct the initial ideal of $I_{K_{2,d}}$ as in Construction
\ref{gb} is a $2 \times d$ matrix consisting of ones.  Consequently,
the graph $H$ constructed from the initial ideal of $I_{K_{2,d}}$ 
will be a bipartite graph with bipartition $V_H = \{e_2,\ldots,e_d\}
\cup \{f_1,\ldots,f_{d-1}\}$ with edge set $E_H = 
\{e_if_j ~|~ 1 \leq j < i \leq d\}$.  (For simplicity here, we have changed our
notation for our vertices using $\{e_2,\ldots,e_d,f_1,\ldots,f_{d-1}\}$
instead of $\{e_{1,2},\ldots,e_{1,d},e_{2,1},\ldots,e_{2,d-1}\}$).
By the next result, finding the $\beta_{i,j}(I_{K_{2,d}})$'s is equivalent to
finding the Betti numbers of $I(H)$:

\begin{theorem}\label{comparebetti}
Fix some integer $d \geq 2$, and consider
the toric ideal $I_{K_{2,d}}$.
Then $I_{K_{2,d}}$ has a linear resolution, and furthermore,
\[\beta_{i,i+2}(I_{K_{2,d}}) = \beta_{i,i+2}(I(H)) ~~\mbox{for all $i \geq 0$}\]
where $H$ is the graph with edge ideal $I(H) = {\rm in}_<(I_{K_{2,d}})$
with $<$ as in Construction \ref{gb}.
\end{theorem}

\begin{proof}
The graph $H$ is co-chordal, so by Theorem \ref{co-chordal}, the
quadratic monomial ideal $I(H)$ has $2 \leq {\rm reg}(I(H)) \leq 2$.
This implies that $I(H)$ has a linear resolution, 
i.e., $\beta_{i,j}(I(H)) = 0$ for $j \neq i+2$.  
Now apply Theorem \ref{initialidealbound}.
\end{proof}

To compute the graded Betti numbers of the edge ideal $I(H)$,
we need the next lemma.

\begin{lemma}[{\cite[Prop. 1.2]{RVT}}]\label{linearstrand}
Let $G$ be a finite simple graph with edge ideal $I(G)$.
Then
\[\beta_{i,i+2}(I(G)) = \sum_{S \subseteq V, |S|= i+2} (\#{\rm comp}(G_S^c)-1)
~~\mbox{for all $i \geq 0$.}\]
Here, $\#{\rm comp}(-)$ is the number of connected components, and $G^c$
denotes the complement.
\end{lemma}

Putting together the above pieces, we arrive at the following formula.

\begin{theorem}
For all $d \geq 2$, the graded Betti numbers of the toric ideal
$I_{K_{2,d}}$ are
$\beta_{i,j}(I_{K_{2,d}}) = 0$ if $j\neq i+2$, and 
\[\beta_{i,i+2}(I_{K_{2,d}}) = 
\sum_{\ell=1}^{i+1}\sum_{r=0}^{d-2-\ell} \binom{\ell-1+r}{\ell-1}\binom{d-\ell-r}{i+2-\ell}
~~\mbox{for all $i \geq 0$.}\]
\end{theorem}

\begin{proof}
Fix some $d \geq 2$.
By Theorem \ref{comparebetti}, it suffices to compute $\beta_{i,i+2}(I(H))$ for all
$i \geq 0$.    To compute these numbers, we can use Lemma
\ref{linearstrand}.  

Fix an $i \geq 0$.  For each $S \subseteq
V$ with $|S|=i+2$, we wish to compute $\#{\rm comp}(H_S^c)$.  Since
$H$ has  bipartition $V  = E \cup F$,  
if $S \subseteq E$ or if $S \subseteq F$, then $H_S^c$ is a complete
graph, and hence connected. Thus $S$ does not contribute to 
$\beta_{i,i+2}(I(H))$
because $\#{\rm comp}(H_S^c) -1 =0$.

So, we can assume that $S \subseteq V$, $|S| = i+2$,
$S \cap E \neq \emptyset$,
and $S \cap F \neq \emptyset$. Because all the $e_i$ vertices will be
adjacent in $H_S^c$, and similarily, all the $f_j$ vertices will be adjacent
in $H_S^c$, we see that $H_S^c$ is connected 
if and only if there exits an $e_i,f_j \in S$ with $i \leq j$.   As we noted
above, if $H_S^c$ is 
connected, it does not contribute to $\beta_{i,i+2}(I(H))$.
To summarize,  we need to count the number of subsets of 
$S \subseteq V$ that have the property that 
$|S| = i+2$, $S \cap E \neq \emptyset$,
$S \cap F \neq \emptyset$, and if $f_j \in S$, there is no
$e_i \in S$ with $i \leq j$. Note if $S$ satisfies these conditions,
then $\#{\rm comp}(H_S^c)=2$,
so each such $S$ contributes one to $\beta_{i,i+2}(I(H))$.

Because the set $S$ must contain at least one vertex of $E$,
we can take at most $i+1$ vertices from among the $f_j$'s.  
Fix an integer $1 \leq \ell \leq i+1$.  Then 
for each $r=0,\ldots,d-2-\ell$,
there are $\binom{\ell-1+r}{\ell-1}$ subsets $T$ of $\{f_1,\ldots,f_{d-1}\}$
with $|T| = \ell$, with $f_{\ell+r} \in T$ and for all $f_a \in T$, $a \leq \ell+r.$
To see this, pick a subset of size $\ell-1$ from $\{f_1,\ldots,f_{\ell+r-1}\}$
and then add $f_{\ell+r}$.

Now consider any of the $\binom{\ell-1+r}{\ell-1}$ subsets of size $\ell$
where  $\ell+r$ is the largest index of a vertex among the $f_j$'s.
In order to make a subset $S$ of size $i+2$ from this set
that has the property that $H_s^c$ is not connected,  we then
need to pick $i+2-\ell$ of the remaining vertices among $\{e_{\ell+r-1},\ldots,e_d\}$.
So, we have to pick $i+2-\ell$ things
among $d-r-\ell$ things.   This gives us then $\binom{\ell-1+r}{\ell-1}\binom{
d-r-\ell}{i+2-\ell}$ possible subsets.  Now summing over all $\ell$ and $r$
gives the desired formula.
\end{proof}

\begin{remark}
It is known (e.g., \cite[Proposition 9.1.2]{VBook}) that when
$G = K_{a,b}$, then $I_G$ is generated by the $2 \times 2$ minors 
of a generic $a \times b$ matrix.  One
could deduce the graded Betti numbers of $I_G$ using the Eagon-Northcott resolution (\cite{EN}) for $2\times 2$ minors of generic matrices.  Our proof
gives a combinatorial argument for the graded Betti numbers.
\end{remark}


\end{document}